\documentclass[12pt,a4paper,twoside]{amsart}
\usepackage{amsmath}
\usepackage{amssymb}
\usepackage{amsfonts}

\date{}

\def\Ker{\text{\rm Ker}}

\def\bbr{\mathbb{F}}

\def\bbz{\mathbb{Z}}

\def\bbh{\mathbb{H}}

\def\suml{\sum\limits}
\theoremstyle{plain}
\newtheorem{thm}{Theorem}
\newtheorem*{thm*}{Theorem}
\newtheorem{lemma}[thm]{Lemma}
\newtheorem{prop}[thm]{Proposition}
\theoremstyle{definition} 

\newtheorem*{definition*}{Definition}
\newtheorem*{thm1*}{Theorem A1}
\newtheorem*{thm2*}{Theorem A2}

\newtheorem*{conjecture*}{Conjecture}
\newtheorem{cor}[thm]{Corollary}

\newtheorem*{claim*}{Claim}
\newtheorem*{question*}{Question}
\newtheorem{remark}[thm]{Remark}
\def\bbz{\mathbb{Z}}
\def\bbq{\mathbb{Q}}

\def\bbr{\mathbb{R}}

\def\bbh{\mathbb{H}}
\def\be{\begin{equation}}
\def\ee{\end{equation}}

\def\a{\alpha}
\def\vare{\varepsilon}

\theoremstyle{remark}  

\DeclareMathOperator{\Vol}{Vol}
\DeclareMathOperator{\Dim}{Dim}

\DeclareMathOperator{\Diam}{Diam}

\begin{document}

\title[Quantum codes]{Quantum error correcting codes\\ and 4-dimensional arithmetic hyperbolic manifolds}
\author{Larry Guth}
\author{Alexander Lubotzky}
\maketitle

\baselineskip 16pt

\begin{abstract} Using $4$-dimensional arithmetic hyperbolic manifolds, we construct some new homological quantum error correcting codes. They are LDPC codes with linear rate and distance $n^\epsilon$. Their rate is evaluated via Euler characteristic arguments and their distance using $\bbz_2$-systolic geometry.
This construction answers a queston of Z\'emor \cite{Z}, who asked whether homological codes with such parameters could exist at all.  
\end{abstract}
\section{Introduction}

The goal of this paper is to present homological quantum error correcting codes (QECC) based on a family of finite sheeted congruence covers of a fixed $4$-dimensional arithmetic hyperbolic manifold. We then estimate the rate and distance of these codes and show that they answer a question of Z\'emir \cite{Z}, who asked if homological codes with such parameters could exist at all.

A (classical linear) code is a subspace $C$ of $\bbz^n_2$ of dimension $k$.  Its rate $r$ is, by definition, $r = \frac kn$ and its distance
is defined as $d = d(C) = \min \{ wt(\a) \Big | 0 \neq \a \in C\}$ where $wt(\a)$ is the Hamming weight of $\a$, i.e., the number of non-zero entries of $\a$.  Write $\delta = \delta (C)$ for $\frac dn $ - the normalized distance.  The standard terminology is that $C$ is an $[n, k, d]$-code.

The quantum codes we will consider here will all be the so called CSS-codes (see \cite{NC}, \cite{P}, \cite{F1} for an explanation  what the following construction has to do with quantum error correction).  A quantum CSS-code $\mathcal{C} = (W_1, W_2)$ is defined by two orthogonal subspaces $W_1$ and $W_2$ in $\bbz^n_2$, i.e., for every $\a = (a_1,\dots, a_n) \in W_1$ and $\beta = (b_1,\dots, b_n) \in W_2$, $\a \cdot \beta = 0$ where $\a \cdot \beta := \suml^n_{i = 1} a_ib_i$.  Let $W^\perp_1, W^\perp_2$ be the orthogonal subspaces to $W_1 $ and $W_2$, so: $W_1 \subseteq W^\perp_2$ and $W_2 \subseteq W^\perp_1$.  The dimension $k$ of $\mathcal C$ is defined as $k = \dim(W^\perp_1/W_2) = \dim (W^\perp_2 / W_1) = n - \dim W_1 - \dim W_2$ and the rate $ r = \frac kn$. The distance of the code $\mathcal C$ is defined as:
\[d(W_1, W_2) = \min\{ wt(\a)\Big | \a \in (W^\perp_1\setminus W_2) \cup (W^\perp_2 \setminus W_1)\}\] and $ \delta = \delta (\mathcal C) = \frac dn$.  Here one writes that $\mathcal C$ is an $\big[ [n, d, c] \big ]$-code.

\smallskip

While studying codes, one is usually interested in a family of codes when $n \to \infty$.  By abuse of the language we say that $C$ (resp. $\mathcal C$) is a good code if $r$ and $\delta$ are both bounded away from $0$, or equivalently, the dimension and the distance grow linearly with the length $n$ of the code.

The code $C$ (resp. $\mathcal C$) is called LDPC (= low density parity check) if $C^\perp$ is spanned by vectors of bounded Hamming weight (resp. $W_i = (W^\perp_i)^\perp$ is spanned by vectors of bounded weight for $i = 1 $ and $2$).

Good LDPC classical codes have been known to exist, by random consideration, since the fundamental work of Gallager \cite{G} in the 60's.  In 1996, Sipser and Spielman made an explicit construction of such codes using expander graphs (\cite{SS}).  It is still an open problem if good LDPC quantum codes exist at all.

Manifolds and simplicial complexes offer a natural construction of quantum LDPC codes, the so called ``homological codes'',  in the following way:

Let $X$ be a finite simplicial complex of dimension $D$ (if $M$ is a manifold, one can replace $M$ by a triangulation $X$ of it), i.e., $X$ is a set of subsets of $X^{(0)}$ (the set of vertices of $X$) of size $\le D + 1$ with the property that if $F \in X$ and $G \subseteq F$ then $G\in X$.  Let $X^{(i)} $ be the set of subsets in $X$ of size $i + 1$.  The space of mod 2 $i$-chains, $C_i = C_i (X, \bbz_2)$, is the $\bbz_2$-vector space spanned by $X^{(i)}$.  The space of mod 2 $i$-cochains, $C^i = C^i (X, \bbz_2)$, is the space of functions from $X^{(i)} $ to $\bbz_2$.  It is convenient to identify $C_i$ and $C^i$ in the clear way.

Let $\partial_i: C_i \to C_{i-1}$ be the boundary map, i.e.,
\[\partial_i(F) = \suml_{\scriptstyle G<F \atop \scriptstyle |G|=i} G \; \; {\rm  for} \; \;  F\in X^{(i)}\] and its adjoint $\delta_i: C_i\to C_{i+1}$ the cobounding map
\[\delta_i(F) = \suml_{ \scriptstyle F<G \atop \scriptstyle |G|=i+2}G \] (recall that as we work in characteristic two, we can ignore the orientation).  It is well known and easy to prove that for all $i$
\begin{equation}\label{1.1} \partial_i \circ \partial_{i+1} = 0 \; \; \hbox{\rm and} \; \; \delta_{i} \circ \delta_{i-1} = 0.\end{equation}
Hence $B_i:= $ the $i$-boundaries $ = Im \partial_{i+1}$ (resp: $B^i = $ the $i$-coboundaries $ = Im \delta_{i-1}$) is contained in $Z_i = $ the $i$-cycles $= \Ker \partial_i $ (resp: $Z^i = $ the $i$-cocycles = $\Ker \delta_i$).

Moreover, one can easily check that (1) implies the following (which is crucial for the quantum codes application):
\begin{equation}\label{1.2}
B^\perp_i = Z^i \; \; {\rm and} \; \; (B^i)^\perp = Z_i.
\end{equation}

One can therefore associate with $X$, for every $i \le D$, a quantum CSS-code $\mathcal C = (B_i, B^i)$ whose length is $n = |X^{(i)}|$, and its dimension is $k = \dim Z_i/B_i = \dim H_i (X, \bbz_2)$, i.e., the dimension of the homology of $X$ (or of $M$) with coefficients in $\bbz_2$.  The latter is also $k  = \dim Z^i/B^i = \dim H^i (X, \bbz_2)$, i.e., the dimension of the $i$-cohomology group.  Finally, the distance $d = d(\mathcal C)$ is the Hamming weight of a non-trivial homology or cohomology class i.e., the minimum weight of an $i$-cycle (which is not an $i$-boundary) or $i$-cocycle (which is not an $i$-coboundary).

Homological codes are attractive as if one let $M$  varies over finite sheeted covers of a fixed compact manifold, then the codes obtained are automatically LDPC since $B_i $ (resp. $B^i$) is generated by the images of the cells of dimension $i + 1$ (resp. $i - 1 $).  This gives these codes great potential, and this is what brings us to systolic geometry.

Systolic geometry is a subarea of Riemannian geometry studying volumes of non-trivial cycles and cocycles.  Fairly recently, it has been noticed that there is a connection between quantum error correcting codes and systolic geometry with $\bbz_2$ coefficients (cf. \cite{Fr}, \cite{MFL}, \cite{Z}, \cite{F2}).

Some of the most studied quantum codes, e.g. the toric codes and surface codes (cf. \cite{Z} and \cite{F1}), can be considered as special cases of homological codes. But, while they are LDPC they are all far from being good.

 Z\'emor \cite{Z}, based on the known examples and on intuition coming
 from graph theory, made the following suggestion:

\begin{question*}[Z\'emor \cite{Z} ] Let $\mathcal C$ be an $\big[[n, k, d]\big]$ homological quantum code. Is it always true that $kd^2\le n^{1 + o(1)}$? (or, in the notation of $r = \frac kn$ and $\delta = \frac dn$, $r\delta^2 \le n^{-2+o(1)}$)

\end{question*}

In \cite{F2}, Fetaya essentially proved that this is indeed true for codes coming from 2-dimensional surfaces.  We prove that it is not the case for codes coming from 4-dimensional manifolds.  

\begin{thm} \label{main} There exist $\vare, \vare', \vare'' > 0$ and a sequence of 4-dimensional hyperbolic manifolds $M$ with triangulations $X$ such that the associated homological quantum codes constructed in $C^2(X, \bbz_2) = C_2 (X, \bbz_2)$ are $\big[[ n, \vare' n, n^{\vare{''}}]\big] $ codes and so:
$$ kd^2\ge n^{1+\vare}.$$
\end{thm}

Theorem 1 will be deduced from the following geometric result.

\begin{thm}  There is a constant $\epsilon > 0$, a closed hyperbolic 4-manifold $M_0$ and a sequence of finite sheeted covers
$M_j \rightarrow M_0$ with $\Vol_4 M_j \rightarrow \infty$ obeying the following estimates:

1. The dimension of $H_2(M_j, \mathbb{Z}_2)$ is $\ge (1/100) \Vol_4 M_j$.

2. Every homologically non-trivial mod 2 2-cycle in $M_j$ has area $\ge (\Vol_4 M_j)^{\epsilon}$.

\end{thm}

A crucial ingredient in the proof is a theorem of Anderson \cite{A} which says that every homologically non-trivial $i$-cycle in a $D$-dimensional hyperbolic manifold has volume at least the volume of a hyperbolic $i$-ball of radius $R$, where $R$ is the injectivity radius of the manifold. Note that when $i=1$, this gives a linear bound on the volume, while for $i > 1$ it gives an exponential bound.  So we get much better lower bounds when $i > 1$.
For homological quantum codes, one needs lower bounds on both $i$-cycles and $i$-cocycles.  By Poincar\'e duality (see Section \ref{codesyst}), one wants lower bounds on $i$-cycles and $(D-i)$-cycles.  That is why we chose to work with $D=4$ and $i=2$, which is the smallest dimension where $i$ and $D-i$ are more than 1.

We should note however that the existence of LDPC quantum codes with parameters $\big[[n, \vare' n, n^{\vare{''}}]\big]$ is not new.  In [TZ], Tillich and Z\'emor discovered such codes with parameters $\big[[n, \vare' n, n^{0.5}]\big]$.  In fact, for the codes we construct, $\epsilon'' \le 0.3$ - see Remark \ref{3/10bound}.  The point of the current work is that our codes are homological.  This brings back the systolic geometry  to the race of finding good LDPC quantum codes, if such codes really exist.

The paper is organized as follows.  In Section 2, we describe the connection between distance of codes and systoles.  We deduce Theorem 1 from Theorem 2, postponing the proof of Theorem 2 to the next two sections.  In Section 3, we estimate the dimension of the homology groups using Euler characteristic arguments.  In Section 4, we estimate the sizes of cycles in congruence covers.  Along the way, we give lower bounds on the injectivity radius of the congruence covers.  This section is making a crucial use of results of Anderson \cite{A} in hyperbolic geometry.  As we expect this paper to have a diverse audience, we tried whenever we can to elaborate a bit on some methods even when they are well known to experts.  In particular, in Section 5 we explain the basic idea behind Anderson's theorem.

{\bf Acknowledgements.} The authors acknowledge support from the Sloan foundation, NSF, ERC and ISF.

\section{Codes and systoles} \label{codesyst}

Suppose that $(M^n, g)$ is a closed Riemannian manifold equipped with a triangulation $X$.  In the introduction, we recalled the $CSS$ code corresponding to i-dimensional chains in $X$: this is the code $(B_i, B^i)$.  We would like to understand how the 
distance of this code is related to the geometry of $(M, g)$.  

The $i$-dimensional systole of $(M^n, g)$ with coefficients in $\bbz_2$ is the infimal volume of a homologically non-trivial Lipschitz $i$-cycle in $(M^n, g)$.  We denote it by $Sys_i(M^n, g)$.  Below, we will briefly recall what these words mean.  The systole is a quantity from Riemannian geometry which is analogous to the distance of a code.  In particular, we recall the following result connecting systoles and codes.

\begin{prop} \label{distsyst} Suppose that $M_0^D$ is a closed $D$-dimensional manifold equipped with a Riemannian metric $g_0$ and a triangulation $X_0$.  Let $M \rightarrow M_0$ be a finite sheeted cover, and let $g$ and $X$ be the pullbacks of $g_0$ and $X_0$.  Then the distance of the code $(B_i(X), B^i(X))$ obeys the inequality

\begin{equation}
 d( B_i(X), B^i(X)) \ge c(M_0, g_0, X_0) \min \left( Sys_i (M, g), Sys_{D-i}(M,g) \right),
 \end{equation}
 
 where $c(M_0, g_0, X_0)$ is a constant.

\end{prop}

Let us now recall Lipschitz chains and cycles, to clarify the definition of the systole.
A Lipschitz $i$-chain with coefficients in $\bbz_2$ is a finite sum $\sum_j a_j f_j$ where $a_j \in \bbz_2$ and $f_j$ is a Lipschitz map from the standard $i$-simplex $\Delta^i$ to $M$.  A Lipschitz 1-chain is a bunch of parametrized curves in $M$, and a Lipschitz $i$-chain is a bunch of parametrized $i$-simplices.  We denote the Lipschitz $i$-chains by $C_{i, Lip} (M, \bbz_2)$.  There is a boundary map $\partial_{i, Lip}:C_{i, Lip} (M, \bbz_2) \rightarrow C_{i-1, Lip}(M, \bbz_2)$, defined by restricting each map $f_j$ to the $(i-1)$-simplices in the boundary of $\Delta^i$.  The Lipschitz $i$-cycles form the kernel of $\partial_{i, Lip}$.  Each $i$-cycle in $Z_i(X, \bbz_2)$ can be considered as a Lipschitz cycle, but most Lipschitz cycles in $(M^n, g)$ do not come from any element of $Z_i(X, \bbz_2)$.  The reader can visualize the $i$-cycles of $Z_i(X, \bbz_2)$ as surfaces made from the $i$-faces of $X$, whereas Lipschitz $i$-cycles do not have to lie in the $i$-skeleton of $X$.  

A standard result of topology says that the homology of the chain complex of Lipschitz chains is the same as the homology $H_i(X, \bbz_2)$.  In particular, an $i$-cycle in $Z_i(X, \bbz_2)$ is homologically non-trivial if and only if the corresponding Lipschitz cycle is homologically non-trivial.  

The volume of a Lipschitz map $f_j: \Delta^i \rightarrow (M,g)$, is defined to be the volume of the pullback metric $f_j^*(g)$ on $\Delta^i$.  We can also think of it as the $i$-dimensional volume of $f_j(\Delta^i)$, counted with multiplicity if the image covers some $i$-dimensional surface multiple times.  The volume of the Lipschitz chain $\sum_j a_j f_j$ is $\sum_j |a_j| \Vol_i f_j$.  

Now suppose that $\alpha$ is a chain in $C_i(X, \bbz_2)$.  Since we are working with mod 2 coefficients, we can abuse notation and think of $\alpha$ as a subset of the $i$-dimensional faces of $X$.  The weight of $\alpha$ is just the number of $i$-faces in $\alpha$.  
The volume of the Lipschitz chain corresponding to $\alpha$ is the sum of the volumes of the $i$-faces in $\alpha$.

Therefore, we get the following inequalities between the weight of $\alpha \in Z_i(X, \bbz_2)$ and the volume of $\alpha$ as a Lipschitz cycle.

$$ (\min_{F \subset X} \Vol_i F) wt(\alpha) \le \Vol \alpha \le (\max_{F \subset X} \Vol_i F) wt(\alpha). $$

\noindent In these formulas, the maximum or minimum is over all the $i$-faces of $X$.  

In Proposition \ref{distsyst}, we start with a closed manifold $(M_0, g_0)$ with a triangulation $X_0$.  Then we consider finite sheeted covers $M \rightarrow M_0$ with pullback metric $g$ and pullback triangulation $X$.  The maximum and minimum volumes of 
$i$-faces in $X$ are the same as in $X_0$, so they are uniformly bounded.  Therefore, the volume and weight of $\alpha$ agree up to a constant factor.  These observations prove the following lemma.

\begin{lemma} Suppose that $M \rightarrow M_0$ is a finite sheeted cover.  Suppose that $M_0$ is equipped with a metric $g_0$ and a triangulation $X_0$, and let $g$ and $X$ be the pullbacks.  Then for any $\alpha \in Z_i(X, \bbz_2)$, we have

$$wt(\alpha) \ge c_1(M_0, g_0, X_0) \Vol_i \alpha.$$

\end{lemma}

As a corollary, we can bound the minimal weight of a homologically non-trivial cycle $\alpha$ in terms of the systole of $(M_i, g_i)$.

\begin{lemma} Suppose that $M \rightarrow M_0$ is a finite sheeted cover.  Suppose that $M_0$ is equipped with a metric $g_0$ and a triangulation $X_0$, and let $g$ and $X$ be the pullbacks.  Then the minimal weight of any $\alpha$ in $Z_i \setminus B_i$ is at least $c_2(M_0, g_0, X_0) Sys_i (M, g)$. 
\end{lemma}

This lemma is the first half of the proof of Proposition \ref{distsyst}.  To finish the proof, we have to give a lower bound for the weights of $\alpha$ in $Z^i \setminus B^i$.  We can do this using Poincar\'e duality for the manifold $M$.

Suppose that $M^D$ is a closed $D$-dimensional manifold with triangulation $X$.  Then there is a Poincar\'e dual polyhedral structure $X'$ on $M$.  There is a vertex of $X'$ in the center of each $D$-simplex of $X$.  There is an edge of $X'$ through the center of each $(D-1)$-face of $X$.  The $(D-1)$-face borders exactly two $D$-faces, and the edge goes from the center of one to the center of the other.  More generally, there is an $i$-face of $X'$ through the center of each $(D-i)$-simplex in $X$.  (For a description of dual polyhedral structures and Poincar\'e duality, see Chapter 5.2 of \cite{Sha}.)

Since each $i$-dimensional face of $X$ corresponds to a unique $(D-i)$-dimensional face of $X'$, we get a Poincar\'e duality isomorphism $P: C^i(X, \bbz_2) \rightarrow C_{D-i}(X', \bbz_2)$.  The key feature of Poincar\'e duality is that the coboundary map on cochains in $X$ corresponds to the boundary map on chains in $X'$.  In formulas, this means that for each $\alpha \in C^i(X, \bbz_2)$, 

$$P (\delta_i \alpha) = \partial_{D-i} (P \alpha). $$

In particular, Poincar\'e duality maps cocycles in $X$ to cycles in $X'$ and coboundaries in $X$ to boundaries in $X'$.  Therefore,
the minimum weight of any $\alpha \in Z^i(X, \bbz_2) \setminus B^i(X, \bbz_2)$ is the same as the minimal weight of any $\alpha$ in
$Z_{D-i}(X', \bbz_2) \setminus B_{D-i}(X', \bbz_2)$.
Therefore, we get the following inequalities between the weight of $\alpha \in Z^i(X, \bbz_2)$ and the volume of $P \alpha$ as a Lipschitz cycle:

$$ (\min_{F \subset X'} \Vol_{D-i} F) wt(\alpha) \le \Vol_{D-i} P \alpha \le (\max_{F \subset X'} \Vol_{D-i} F) wt(\alpha). $$

\noindent In these formulas, the maximum or minimum is over all the $(D-i)$-faces of $X'$.  

Now suppose again that $M$ is a finite sheeted cover of a closed manifold $M_0$, where $M_0$ is equipped with metric $g_0$ and triangulation $X_0$.  We let $X_0'$ be the Poincar\'e dual polyhedral structure for $X_0$.  We let $g$ be the pullback of $g_0$, $X$ be the pullback of $X_0$, and $X'$ be the pullback of $X_0'$.  Now $X'$ is still Poincar\'e dual to $X$.  Moreover, the maximum and minimum volumes of faces in $X'$ are the same as in $X_0'$.  

Therefore, we see that the weight of any $\alpha \in C^i(X, \bbz_2)$ is at least $c(M_0, g_0, X_0, X_0') \Vol_{D-i} P \alpha$.  
Next suppose that $\alpha \in Z^i(X, \bbz_2) \setminus B^i(X, \bbz_2)$.  Since Poincar\'e duality respects boundaries, it follows that
$P \alpha \in Z_i(X', \bbz_2) \setminus B_i(X', \bbz_2)$.  Therefore, 
the minimal weight of any $\alpha \in Z^i(X, \bbz_2) \setminus B^i(X, \bbz_2)$ is at least $c(M_0, g_0, X_0, X_0') Sys_{D-i}(M, g)$.  

This finishes the proof of Proposition \ref{distsyst}.  Proposition \ref{distsyst} is a bridge connecting geometric properties of Riemannian manifolds and codes.  Using this bridge, we can build interesting codes from interesting towers of Riemannian manifolds.  In particular we will prove the following result about congruence covers of hyperbolic 4-manifolds.  

\begin{thm}  \label{hypmain} There is a constant $\epsilon > 0$, a closed hyperbolic 4-manifold $M_0$ and a sequence of finite sheeted covers
$M_j \rightarrow M_0$ with $\Vol_4 M_j \rightarrow \infty$ obeying the following estimates.

1. The dimension of $H_2(M_j, \mathbb{Z}_2)$ is $\ge (1/100) \Vol_4 M_j$.

2. $Sys_2(M_j) \ge (\Vol_4 M_j)^\epsilon$.  In other words, every homologically non-trivial mod 2 2-cycle in $M_j$ has area $\ge (\Vol_4 M_j)^{\epsilon}$.

(We use the hyperbolic metric on each $M_j$ pulled back from $M_0$.)

\end{thm}

Using this theorem, we can quickly construct codes proving Theorem \ref{main}.  We fix a triangulation $X_0$ of $M_0$.  We let
$X_j$ be the pullback triangulation of $M_j$.  Then we consider the quantum code $(B_2(X_j, \bbz_2), B^2(X_j, \bbz_2))$.

In this code, the spaces $B_2, B^2$ are subspaces of $C_2(X_j, \bbz_2)$.  This is a vector space whose dimension $n$ is equal to the number of 2-faces in $X_j$.  We let $D_j$ be the degree of the cover $M_j \rightarrow M_0$.  The number of 2-faces in $X_j$ is $D_j$ times the number of 2-faces in $X_0$.  Up to a factor $C(M_0, X_0)$, $n$ is equal to $D_j$.  Also, up to a constant factor, $D_j$ is equal to $\Vol M_j$.  So $n \le c_1 \Vol M_j$.

The dimension of the code is $k = \Dim H_2(X_j, \bbz_2) = \Dim H_2(M_j, \bbz_2) \ge 1/100 \Vol M_j$.  Therefore, $k \ge c_2 n$.  In other words, these codes have a linear rate.

Finally, Proposition \ref{distsyst} implies that the distance of the code is at least $c_3 Sys_2(M_j) \ge c_4 n^\epsilon$.  

We will prove Theorem \ref{hypmain} in the next sections using hyperbolic geometry.

\section{Euler characteristic of Hyperbolic manifolds}

The Euler characteristic of a hyperbolic manifold can be computed by the Gauss-Bonnet-Chern theorem.

\begin{thm} (Allendoerfer-Weil, Chern \cite{C}) If $M^{2n}$ is a closed oriented manifold with Riemannian metric $g$, then
the Euler characteristic of $M$ is given by $c_n \int_M Pf(K_g) dvol_g$, where $c_n > 0$ is a dimensional constant, and $Pf(K_g)$ 
is the Pfaffian of the curvature of $g$.
\end{thm}

For a hyperbolic 2n-manifold, $Pf(K_{hyp})$ is a constant, and we see that the Euler characteristic of $(M^{2n}, hyp)$ is 
$c'_n Vol_{2n} (M, hyp)$, for some constant $c'_n$.  We will evaluate this constant, and we will see that it is never zero, and that
the sign of $c'_n$ is $(-1)^n$. 

\begin{cor} If $M^{2n}$ is a closed oriented hyperbolic manifold, then the Euler characteristic of $M$ is $(-1)^n \cdot 2  \Vol M / \Vol S^{2n}$.  
\end{cor}

\begin{proof} For the unit sphere, the Gauss-Bonnet-Chern theorem gives $2 = c_n Pf(K_{S^{2n}}) \Vol S^{2n}$.  Therefore,
$c_n Pf(K_{S^{2n}}) = 2 / \Vol S^{2n}$.  The curvature of hyperbolic space is negative the curvature of the unit sphere: $K_{H^{2n}} = - K_{S^{2n}}$.
From the formula for the Pfaffian in \cite{C}, it follows that $Pf(K_{H^{2n}}) = (-1)^n Pf(K_{S^{2n}})$.  
For $(M^{2n}, hyp)$, the Gauss-Bonnet-Chern theorem implies that the Euler characteristic is
$c_n Pf(K_{H^{2n}}) \Vol(M, hyp) = (-1)^n 2 (\Vol S^{2n})^{-1} \Vol (M, hyp)$.
\end{proof}

As a corollary, we see that for a hyperbolic 4-manifold $M$, the dimension of $H_2(M, \bbz_2)$ grows linearly with the volume of $M$.

\begin{cor} Suppose that $M$ is a connected closed hyperbolic 4-manifold with volume $V$. 
Then $\Dim H_2(M, \bbz_2)$ is at least $(2 / \Vol S^4) V - 2$.  
\end{cor}

\begin{proof} By the Corollary above, the Euler characteristic of $M$ is $(2 / Vol S^4) V$.  The Euler characteristic of
$M$ is equal to $\sum_{d=0}^4 (-1)^d \Dim H_d(M, \mathbb{Z}_2)$.  Since $M$ is connected, $H_0(M_i)$ and $H_4(M_i)$ have dimension 1.  Since
the odd dimensions contribute negatively, we get $\Dim H_2(M_i) \ge (2 / Vol S^4) V - 2$.

\end{proof}

For all $V$ sufficiently large, we see that $\Dim H_2(M, \bbz_2) \ge (1/100) V$.  This proves the lower bound on the dimension of $H_2$ in Theorem \ref{hypmain}.

\section{Congruence covers of hyperbolic manifolds}  

Recall that the group of orientation-preserving isometries of hyperbolic space $\bbh^D$ is the connected component of the identity of $SO(D,1)$, which we denote $SO^o(D,1)$.  We will make closed hyperbolic manifolds by quotienting $\bbh^D$ by discrete subgroups of $SO^o(D,1)$.  The discrete subgroups we use will be arithmetic lattices.  
We define arithmetic lattices below, but let us begin with a couple of examples.

The simplest example is $SO(D,1, \bbz) \subset SO(D,1, \bbr)$.  Here we define $SO(D, 1, \bbr) \subset GL(D+1, \bbr)$ as the group of matrices that preserve the quadratic form $-x_0^2 + x_1^2 + ... + x_D^2$ on $\bbr^{D+1}$, and $SO(D, 1, \bbz)$ as the subgroup of matrices that have integer entries.  It turns out that the quotient $SO(D, 1, \bbz) \backslash \bbh^D$ has finite volume but is not compact (cf. Chapter 6 of \cite{Mo}). Since we will use closed hyperbolic manifolds, we need a trickier example.

Let $f$ be the quadratic form $- \sqrt 2 x_0^2 + \sum_{i=1}^D x_i^2$.  Let $SO_f \subset GL(D+1, \bbr)$ be the subgroup of matrices that preserve $f$ and have determinant 1.  The form $f$ has signature $(D,1)$ and so $SO_f$ is conjugate to $SO(D,1)$ in $GL(D+1, \bbr)$.  We fix an isomorphism $SO_f \rightarrow SO(D,1)$ and so we think of the component of the identity, $SO^o(f)$, as the group of orientation-preserving isometries of $\bbh^D$.  

Now we let $SO_f( \bbz [ \sqrt2 ]) \subset SO_f$ be the subgroup of matrices with entries in $\bbz [\sqrt2]$.  The group $SO_f (\bbz [\sqrt2])$ is a discrete cocompact subgroup of $SO_f$.  It is not obvious that $SO_f (\bbz [\sqrt2])$ is either discrete or cocompact.  There is a short elegant proof that it is discrete, which we include.  For a proof that it is cocompact, see Chapter 6 of \cite{Mo}.

The ring $\bbz [\sqrt2] \subset \bbr$ is not discrete.  There are two group homomorphisms from $\bbz [\sqrt2]$ to $\bbr$.  One sends $\sqrt 2$ to $\sqrt 2$, the other sends $\sqrt 2$ to $- \sqrt 2$.  Let $\phi_1$, $\phi_2$ be the two homomorphisms, where $\phi_1( \sqrt 2) = \sqrt 2$ and $\phi_2 (\sqrt 2) = - \sqrt 2$.  Now the image $(\phi_1, \phi_2) \bbz [\sqrt2] \subset \bbr^2$ is discrete.

If we apply the map $\phi_2$ to each coefficient of a matrix $m \in SO_f( \bbz[ \sqrt 2])$, we do not get an element of $SO_f$.  Instead we get an element of $SO_{\tilde f}$ where $\tilde f$ is the quadratic form $+ \sqrt 2 x_0^2 + \sum_{i=1}^D x_i^2$.  
So $\phi_2$ induces a group homomorphism $\phi_2: SO_f(\bbz [\sqrt2 ]) \rightarrow SO_{\tilde f}$.  Combining $\phi_1$ and $\phi_2$, we get an injective homomorphism $SO_f(\bbz [\sqrt2]) \rightarrow SO_f \times SO_{\tilde f}$.  
The image of this homomorphism is discrete.  Moreover, $\tilde f$ has signature $(D+1, 0)$, and so $SO_{\tilde f}$ is conjugate to $SO(D+1)$ and is compact.  Therefore, $SO_f (\bbz [\sqrt2])$ is a discrete subgroup of $SO_f$.

We let $\Gamma_1 = SO_f (\bbz [\sqrt2]) \subset SO_f$, and we think of $SO(f)$ acting on hyperbolic space.  Next we define subgroups $\Gamma_N \subset \Gamma_1$ as follows.  A matrix in $SO_f(\bbz [\sqrt2])$ can be written (uniquely) in the form $A + B \sqrt 2$, where $A, B$ are matrices with integer coefficients.  Such a matrix lies in $\Gamma_N$ if and only if $A = Id$ modulo $N$ and $B = 0$ modulo $N$.  The subgroups $\Gamma_N$ are called principal congruence subgroups of $\Gamma_1$.  The groups $\Gamma_N$ is a normal subgroup of $\Gamma_1$, because it is the kernel of the reduction mod $N$ map from $SO_f(\bbz [\sqrt2])$ to $SO_f( (\bbz / N \bbz) [\sqrt2])$.  

For all sufficiently large $N$, $\Gamma_N$ acts without fixed points on $\bbh^D$, which we will prove below.  Therefore, we can define the hyperbolic manifolds $M_N = \Gamma_N \backslash \bbh^D$ for all sufficiently large $N$.  Taking $D=4$, these hyperbolic manifolds are the examples in Theorem \ref{hypmain}.

The index of $\Gamma_N$ in $\Gamma_1$ is equal to the cardinality of the image of $SO_f(\bbz[\sqrt2])$ in $SO_f ( (\bbz/ N \bbz) [\sqrt2])$.  This cardinality is at most the cardinality of the set of $(D+1) \times (D+1)$ matrices with coefficients in the ring $(\bbz / N \bbz) [\sqrt 2]$, and so it is at most $N^{2 (D+1)^2}$.  A more accurate estimate is $N^{2 \Dim SO_f}$, but we do not need it.

\subsection{Injectivity radius estimates}

Let $M$ be a closed hyperbolic $D$-manifold.  The injectivity radius of $M$ is at least $R$ if and only if, for every point $p \in M$, the metric ball around $p$ with radius $R$ is isometric to the hyperbolic $D$-ball of radius $R$.  Next we need a lower bound for the injectivity radius of our hyperbolic manifolds $M_N = \Gamma_N \backslash \bbh^D$.  
Similar estimates have appeared before in some particular cases in \cite{BS} and \cite{KSV}, and more generally in Section 3.C.6 of \cite{Grsis}).  Let us prove it now directly for our concrete examples.  Later we give a very general estimate.

\begin{prop} \label{injrad} Let $D$, $\Gamma_1, \Gamma_N$ be as above.  Then there are constants $c_1, c_2 > 0$ so that the injectivity radius of $M_N$ is at least $c_1 \log N - c_2 $.
\end{prop}

Since the volume of $M_N$ grows like the index of $\Gamma_N$ in $\Gamma_1$, which grows polynomially in $N$ we get the following corollary.

\begin{cor} Let $D$, $\Gamma_1, \Gamma_N$ be as above.  The there is a constant $c > 0$ so that the injectivity radius of $M_N$ is at least $c \log \Vol M_N$ for all $N$ sufficiently large.
\end{cor}

\begin{proof} Let $\pi: \bbh^D \rightarrow M_N = \Gamma_N \backslash \bbh^D $ be the quotient map.  
Let $R$ be the injectivity radius of $M_N$. Since the injectivity radius of $M_N$ is less than $2R$, there is some $p \in M$ so that the ball of radius $2R$ around $p$ is not isometric to a ball of radius $2R$ in $\bbh^D$.  Let $p'$ be a preimage of $p$ in $\bbh^D$.  Let $B_{\bbh^D}(p', 2R)$ denote the ball around $p'$ of radius $2R$ in $\bbh^D$.  The map $\pi: B_{\bbh^D}(p', 2R) \rightarrow B_M(p, 2R)$ must not be an isometry.  The only way it can fail to be an isometry is that two points of $B_{\bbh^D}(p', 2R)$ lie in the same $\Gamma_N$ orbit.  The distance between these two points must be $< 4 R$.  Therefore, there exists a non-identity element $n \in \Gamma_N$ and a point $x \in \bbh^D$ so that $d(x, nx) < 4 R$.  (Here $d$ denotes the distance in $\bbh^D$.) 

Let $F$ be a fundamental domain for $\Gamma_1$.  Using the symmetries of the situation, we will show that we can arrange for the point $x$ above to lie in $F$.  There exists some $\gamma \in \Gamma_1$ so that $\gamma x \in F$.  Since $\gamma$ acts by isometries, we have

$$ d(x, nx) = d( \gamma x, \gamma (nx) ) = d \left( \gamma x, (\gamma n \gamma^{-1}) (\gamma x) \right). $$

Now we define $x' = \gamma x \in F$ and $n' = \gamma n \gamma^{-1} \in \Gamma_N$, and we see that
$d(x', n' x') < 4R$.  (Here we used that $\Gamma_N$ is normal.)

Next we use the fact that $\Gamma_1$ is cocompact.  This implies that the fundamental domain has finite diameter $\Diam(F)$.  We fix a point $x_0 \in F$, and we are guaranteed that $d(x_0, x') \le \Diam(F)$.  Since $n'$ acts isometrically, $d(n' x_0, n' x') \le \Diam(F)$ as well.  By the triangle inequality we get a bound on $d(x_0, n' x_0)$:

$$ d (x_0, n' x_0) \le 4 R + 2 \Diam(F). \eqno{(1)}$$

Everything we said so far makes sense for any fundamental domain.  It is convenient to choose a particular fundamental domain.  We fix a point $x_0$, and we consider the Dirichlet fundamental domain defined by

$$ F := \{ x \in \bbh^D | d(x, x_0) \le d(\gamma x, x_0) \textrm{ for all } 1 \not= \gamma \in \Gamma_1 \}. $$

Clearly $x_0 \in F$.  Since $\Gamma_1$ is cocompact, $F$ is compact.  The boundary of $F$ is defined to be the set of points $x \in F$ so that $d(x, x_0) = d(\gamma x, x_0)$ for some $1 \not= \gamma \in \Gamma_1$.  We consider the translates $\gamma F$ for $\gamma \in \Gamma_1$.  Two such translates can intersect only at points in the boundary.  We say that $\gamma_1 F$ and $\gamma_2 F$ are adjacent if they intersect.

We can now prove that $\Gamma_1$ is finitely generated and describe a set of generators.  We let $S \subset \Gamma_1$ be the set of $\gamma \in \Gamma_1$ so that $\gamma F \cap F$ is non-empty.  Since $\Gamma_1$ acts properly discontinuously, there are only finitely many points of $\Gamma_1 x_0$ in any ball in $\bbh^D$.  Since $F$ is compact, we see that any compact set intersects only finitely many cells $\gamma F$.  In particular, it follows that $S$ is finite.  Also, $\gamma F \cap F = \gamma (F \cap \gamma^{-1} F)$, and so $\gamma \in S$ if and only if $\gamma^{-1} \in S$.  Let $s_1, s_2, ..., s_T$ be the elements of $S$.

The set $S$ generates $\Gamma_1$.  Indeed, let $\gamma \in \Gamma_1$.  Consider a path from $x_0$ to $\gamma x_0$.  This path is compact, so it intersects only finitely many cells $\gamma F$.  Using the path, we can choose a finite sequence of cells $F, \gamma_1 F, \gamma_2 F, ..., $ where consecutive cells are adjacent and the last cell is $\gamma F$.  By the definition of $S$, $\gamma$ is a finite product $s_{i_1} \circ ... \circ s_{i_w}$, with $s_{i_j} \in S$.

Recall that for each $\gamma \in \Gamma_1$, the word length $w_S(\gamma)$ is the shortest length of a product $s_{i_1} \circ ... \circ s_{i_w}$ which is equal to $\gamma$, with the $s_i \in S$.  We showed above that $w_S(\gamma) < \infty$ for any $\gamma \in \Gamma_1$.  By making the argument more quantitative, we can give an upper bound for $w_S(\gamma)$ in terms of $d(x_0, \gamma x_0)$.

\begin{lemma} \label{wordlength} Let $\Gamma_1$, $S$ be as above.  For any $\gamma \in \Gamma_1$, $w_S(\gamma) \le c_1 (d(x_0, \gamma x_0) + 1)$. 
\end{lemma}

\begin{proof} Since $\Gamma_1$ acts properly discontinuously, any unit ball contains a finite number of points of $\Gamma_1 x_0$.  Since $F$ is compact, any unit ball intersects a finite number of cells $\gamma F$, $\gamma \in \Gamma_1$.  Now since $\Gamma_1$ is cocompact, this upper bound is uniform over all unit balls.  So we may assume that each unit ball in $\bbh^D$ intersects at most $c_1$ cells $\gamma F$.  

Now take a path from $x_0$ to $\gamma x_0$ of length $d(x_0, \gamma x_0)$. Divide this path into $d(x_0, \gamma x_0) + 1$ segments of length at most 1.  Each of these segments intersects at most $c_1$ cells $\gamma F$.  So we can make a sequence of adjacent cells from $F$ to $\gamma F$ using at most $c_1  (d(x_0, \gamma x_0) + 1)$ cells.  By the definition of $S$, we see that $w_S (\gamma) \le c_1  (d(x_0, \gamma x_0) + 1)$.  \end{proof}

This lemma is all that we will use in the proof of the injectivity radius estimate, but for context it is  
useful to be aware of the following more general result.  

\begin{lemma} (See Theorem 3.6 in \cite{Bow}) Suppose that $\Gamma_0$ is a finitely generated group that acts isometrically, properly discontinuously, and cocompactly on a Riemannian manifold $X$.  Let $S$ be a (symmetric) generating set for $\Gamma_0$.  Let $x_0 \in X$.  Then there are constants $0 < c < C$ so that

$$ c d_X(x_0, g x_0) \le w_S(g) \le C d_X (x_0, g x_0) + C. $$

\end{lemma}

Applying Lemma \ref{wordlength} to equation (1), we see that

$$ w_S(n') \le C_1 R + C_2. $$

\noindent   Rearranging the formula, we get a lower bound

$$ R \ge c_1 w_S(n') - C_3. \eqno{(2)}$$

Next we prove a lower bound on the word length $w_S(n')$.

\begin{lemma} Let $D$, $\Gamma_1, \Gamma_N$, and $S$ be as above.  Then there is a constant $c > 0$ so that any non-identity element $n' \in \Gamma_N$ has $w_S(n') \ge c \log N$. 
\end{lemma}

\begin{proof} Let $n' $ be a non-identity element of $\Gamma_N$.  Write $n' = A' + B' \sqrt 2$ where $A', B'$ are matrices with integer coefficients.  Since $n'$ is not just the identity, one of the entries of $A'$ or $B'$ must have norm at least $N-1$.

Let $M_1 = A_1 + B_1 \sqrt 2$ and $M_2 = A_2 + B_2 \sqrt 2$ be matrices, where $A, B$ have integer coefficients.  Write the product as $M = A + B \sqrt 2$.  We write $| M_1 |$ for the maximum of the entries in $A_1$ or $B_1$, and $|M_2|$ and $|M|$ similarly.  Then a direct computation gives $|M| \le C |M_1| |M_2|$.  

Now let $s_1, s_2, ... s_T $ be the matrices in $S$.  Write $s_i = A_i + B_i \sqrt 2$ where $A_i, B_i$ are integer matrices.  Consider a product of $w$ elements of $S$: $g = s_{i_1} \circ ... \circ s_{i_w}$.  Using the product estimate in the last paragraph repeatedly, we get $| g | \le C^{w-1} (\max_{s_i \in S} |s_i|)^w \le C_2^w$.  Applying this argument to $n'$, we see that $N- 1 \le C_2^{w_S(n')}$.  Taking logarithms finishes the proof.
\end{proof}

Plugging into equation (2), we get $R \ge c_1 \log N - C_2. $  This finishes the proof of the Proposition.  \end{proof} 

\begin{remark} The proof of Proposition \ref{injrad} also shows that $\Gamma_N$ acts without fixed points for all $N$ large enough.  During the proof we showed that for any $x \in \bbh^D$ and any $\gamma \in \Gamma_N$, the distance from $x$ to $\gamma(x)$ is at least $c_1 \log N - c_2 $.  If $N$ is sufficiently large, this distance is positive, and so $\Gamma_N$ has no fixed points.
\end{remark}

For what is needed for our construction of quantum codes, the above example $SO_f(\bbz [\sqrt 2])$ and Proposition \ref{injrad} suffices.  But let us take the opportunity to put on record a general result of this kind which is valid for all simple Lie groups.

Recall that if $G$ is a simple Lie group and $\Gamma$ is a discrete subgroup, then $\Gamma$ is an arithmetic subgroup if there exists a number field $k$, a $k$-algebraic group $H$, and an epimorphism $\phi: H(k \otimes_\bbq \bbr) \rightarrow G$ with compact kernel such that $\phi( H(O) )$ is commensurable to $\Gamma$ where $H(O)$ is the $O$-points of $H$ with respect to some fixed embedding of $H$ into $GL(m)$.  ($H(O)$ depends on this embedding only up to commensurability.)  If $I$ is a non-zero ideal of $O$, then $H(O)_I$ is the kernel of the homomorphism from $H(O)$ to $H( O/I )$.  Then $\Gamma(I)$ is defined to be $\phi( H(O)_I ) \cap \Gamma$.  A subgroup of $\Gamma$ is called a congruence subgroup if it contains $\Gamma(I)$ for some (non-zero) ideal $I \subset O$.  The collection of congruence subgroups of $\Gamma$ does not depend on the embedding of $H$ in $GL(m)$.  We can now state:

\begin{prop} \label{injrad2} Let $G$ be a simple Lie group with maximal compact subgroup $K$, and let $X$ be the symmetric space $G/ K$.  Let $\Gamma$ be an arithmetic subgroup of $G$.  Then there exists a constant $c_1 > 0$ so that for every normal congruence subgroup $N \subset \Gamma$, the injectivity radius of $N \backslash X$ is at least $c_1 \log \Vol (N \backslash X) - c_2$.  
\end{prop}

For principal congruence subgroups, the proof is essentially the same as the proof of Proposition \ref{injrad}.  For arbitrary normal congruence subgroups, one should use the argument given in \cite{LL}, page 459, which shows that every normal congruence subgroup is ``very close'' to being a principal congruence subgroup.  It is worth noting that in Proposition \ref{injrad2}, we need the group $N$ to be both normal and a congruence subgroup.

Suppose that $N$ is normal but not congruence.  For every $D$, there exists an arithmetic lattice in $SO(D,1)$ which is mapped onto $\bbz$ (cf. \cite{Mill} or \cite{L2}).  Taking finite cyclic covers induced by the maps $\bbz \rightarrow \bbz / N \bbz$ will give normal covers with a bounded injectivity radius.

Now we consider examples that are congruence subgroups but not normal.  Fix $\gamma \in \Gamma$, an element of infinite order.  We consider a sequence of smaller and smaller ideals $I \subset O$.  For each $I$, we let $\pi_I: H(O) \rightarrow H(O/I)$.  Then we let $\Lambda_I$ be $\pi_I^{-1} ( \langle \pi_I(\gamma) \rangle )$.  In words, $\Lambda_I$ is the preimage of the subgroup generated by $\pi_I (\gamma)$.  Clearly $\Lambda_I$ contains $\gamma$ for all $I$.  Therefore, $\Lambda_I \backslash X$ contains a non-contractible loop of length independent of $I$.  But the $\Lambda_I$ are all congruence subgroups of $\Gamma$, and the index of $\Lambda_I$ in $\Gamma$ goes to infinity like a power of $[O : I]$.

\subsection{Anderson's systolic bound}

Finally, we need an important result from hyperbolic geometry that connects the systoles of a hyperbolic manifold and its injectivity radius. 

\begin{thm} (Anderson, \cite{A}) Let $(M^D, hyp)$ be a closed manifold with a hyperbolic metric.  Let $Z^i \subset M$ be
a homologically non-trivial $i$-cycle with coefficients in $\mathbb{Z}_2$.  Let $R$ be
the injectivity radius of $(M, hyp)$.  Then the volume of $Z$ is at least the volume of
a ball of radius $R$ in the $i$-dimensional hyperbolic space.  In particular, if $i \ge 2$ and $R \ge 1$, then
$Vol_i (Z) \ge c(i) e^{(i-1) R}$.  
\end{thm}

\begin{remark} This theorem holds for any dimension $i$ in the range $1 \le i \le D-1$.  But the main interest is in the range $2 \le i \le D-1$.  
If $i=1$, then a 1-dimensional hyperbolic ball of radius $R$ is just an interval $[-R, R]$, and it has 1-dimensional volume $2R$.  But for
each $i \ge 2$, the $i$-volume of an $i$-dimensional hyperbolic ball of radius $R$ grows exponentially in $R$, roughly like $e^{(i-1) R}$.  
\end{remark}

\begin{remark} Anderson's original results are stated for integral cycles, but essentially the same proof works also with coefficients in $\bbz_2$ - see Section 5 where we discuss some of the main ideas of the proof.
\end{remark}

Now we have all the tools to prove Theorem \ref{hypmain}.

\begin{proof} Let $M_1$ and $M_N$ be as above.  By Proposition \ref{injrad}
 the injectivity radius of $M_N$ is $R_N \ge c \log \Vol M_N$.  By Anderson's theorem, every homologically non-trivial mod 2 2-cycle in $M_N$ has
area $\ge c e^{R_N} \ge c \Vol M_N^{\epsilon}$.  

On the other hand, using the Gauss-Bonnet-Chern theorem, we already proved that 
$\Dim H_2(M_N) \ge (2 / Vol S^4) \Vol M_N - 2$.  By taking $N$ sufficiently large, we get $\Dim H_2(M_N) \ge (1/100) \Vol M_N$ as desired.

\end{proof}

\begin{remark} \label{3/10bound} The proof shows that the distance of our code is at least $n^{\epsilon''}$ for some $\epsilon'' > 0$.  On the other hand, the distance is at most $O(n^{0.3})$.  This is because our lattice $\Gamma_N$ contains an arithmetic lattice $\Lambda_N$ of $SO(2,1)$.  (In fact,  every arithmetic lattice in $SO(4,1)$ contains an arithmetic lattice of $SO(2,1)$ - see \cite{L}.)  Hence every $M_N = \Gamma_N \backslash \bbh^4$ contains the surface $\Lambda_N \backslash \bbh^2 $.  While the volume of $\Gamma_N \backslash \bbh^4$ grows like $N^{20}$, the area of $\Lambda_N \backslash \bbh^2$ grows like $N^6$.  See Proposition 3.2 in \cite{B} for more details of a similar argument.  
\end{remark}

\section{On Anderson's Theorem}

The proof uses minimal surface theory.  A complete proof (including all the underlying results from minimal surface theory) is somewhat
long, but we can explain the basic ideas behind the proof.  We will explain first the original proof with $\bbz$ coefficients, and at the end we discuss what to do in the $\bbz_2$ case.

The first main idea is to replace $Z$ with a surface of minimal area $Z_{min}$ in its homology class.  A serious result from minimal surface theory
is that there is a kind of generalized surface - a stationary integral current - in the homology class of $Z$ with minimal volume.  
An  integral current is a generalization of an integral chain which can be somewhat more singular.  It suffices to prove a lower bound for $Vol_k Z_{min}$.  Since $Z$ is homologically non-trivial, $Z_{min}$
is not empty.  Minimal surfaces have special geometric properties that can be used to estimate the volume of $Z_{min}$.  

The proof is based on the hyperbolic generalization of the monotonicity theorem for minimal surfaces.  For context we first recall the
standard monotonicity theorem for minimal surfaces in $\mathbb{R}^n$.  

\begin{thm} \label{mono} (Monotonicity) Let $X$ be an integral k-chain in $\mathbb{R}^n$.  Suppose that $X$ has the smallest volume of any chain with
boundary $\partial X$.  Suppose that $\partial X$ lies in $\partial B^n(R)$.  Then the ratio

$$ \frac {\Vol_k (X \cap B^n(r))}{ \Vol_k B^k(r)}$$

\noindent is non-decreasing for $0 < r \le R$.  
\end{thm}

In this formula, $B^k(r)$ is the Euclidean k-ball of radius $r$.  If $X$ contains 0, and if $X$ is a smooth manifold in a small
neighborhood of 0, then the volume ratio goes to 1 as $r \rightarrow 0$.  Then the monotonicity theorem guarantees that the volume ratio is $\ge 1$ for all $r$.  So we see that if $0$ is a smooth point of $X$, then $\Vol_k X$  is at least the volume of the Euclidean k-ball of radius $R$.

We will explain the idea of the proof of the monotonicity theorem below.  Now we give the hyperbolic analogue.

\begin{thm} \label{hypmono} (Anderson \cite{A})  Let $X$ be an integral k-chain in the hyperbolic space $\mathbb{H}^n$.  
Suppose that $X$ has the smallest volume of any chain with boundary $\partial X$.  Let $B_{hyp}^n(R) \subset \mathbb{H}^n$ denote the hyperbolic
ball of radius $R$.  Suppose that $\partial X$ lies in 
$\partial B^n(R)$.  Then the ratio

$$ \frac {Vol_k (X \cap B_{hyp}^n(r))}{ Vol_k B_{hyp}^k(r)}$$

is non-decreasing for $0 < r \le R$.  

\end{thm}

In particular, if $X$ contains the center of the ball $B^n_{hyp}(R)$ and $X$ is a smooth manifold near that point, then $Vol_k X$ is at
least the volume of a k-dimensional hyperbolic ball of radius $R$.

Anderson's theorem is deduced quickly from Theorem \ref{hypmono}.  The stationary integral current $Z_{min}$ as before can have singularities, but a fundamental result of geometric measure theory says that almost every point is regular, i.e. it has a neighborhood where $Z_{min}$ is a smooth submanifold.  We pick one such point
$x$.  We consider the intersection of $Z_{min}$ with the ball around $x$ of radius $R$ equal to the injectivity radius of $(M, hyp)$.  This ball is
isometric to the ball of radius $R$ in $\mathbb{H}^n$.  By the monotonicity theorem, the volume of $Z_{min}$ in this ball is at least the volume
of a k-dimensional hyperbolic ball of radius $R$.

Let us now give the main idea of the proof of Theorems \ref{mono} and \ref{hypmono}.

The monotonicity theorem is based on the formula for the volume of a cone.  Let $Y^{k-1}$ be a surface in $\partial B^n(r) \subset \mathbb{R}^n$.
We let $CY$ denote the cone over $Y$ with vertex at the origin.  This is the union of all the line segments with one endpoint at the origin (the
center of $B^n(r)$) and the other endpoint on $Y$.  The formula for the volume of a cone is as follows:

$$ \Vol_k CY = (r/k) \Vol_{k-1} (Y). $$

Here is one way of thinking about this formula.  Because of the symmetry of the sphere and the ball, we must have a formula of the type
$\Vol_k CY = Const(r,k) \Vol_{k-1} Y$.  Then we can evaluate $Const(r,k)$ by looking at the simplest example: when $CY$ is a k-dimensional ball of
radius $r$ and $Y$ is a (k-1)-dimensional sphere of radius $r$.  Let $VB_k(r)$ be the volume of a k-dimensional ball of radius $r$ and $VS_{k-1}(r)$
be the volume of a (k-1)-sphere of radius $r$.  We get

$$ Vol_k CY = \frac{VB_k(r)}{VS_{k-1}(r)} Vol_{k-1} Y. \eqno{(1)}$$

\noindent The ratio $VB_k(r) / VS_{k-1}(r)$ is equal to $r/k$ because the volume of a $k$-dimenional cone is $(1/k)$ times the volume of the base times the height.  But in fact Equation (1) is more useful than the formulation with $r/k$.

We return to the monotonicity formula.  We let $X \subset B^n(R)$ be our minimal chain.  We let $X_r := X \cap B^n(r)$ and we let $Y_r = \partial
X_r$.  Since $X$ is minimal, $X_r$ must also be minimal, and so

$$ Vol_k X_r \le Vol_k CY_r =  \frac{VB_k(r)}{VS_{k-1}(r)} Vol_{k-1} Y_r.  \eqno{(2)} $$

The coarea inequality says that $\frac{d}{dr} Vol_k X_r \ge Vol_{k-1} Y_r$ (for almost every $r$).  We let $V(r) := Vol_k X_r$, and the last equation becomes the following differential inequality.

$$ V'(r) \ge \frac{VS_{k-1}(r)}{VB_k(r)} V(r). \eqno{(3)}$$

Since $\frac{d}{dr} VB_k(r) = VS_{k-1}(r)$, Equation $(3)$ implies:

$$ \frac{d}{dr} \left( \frac{V(r)}{VB_k(r)} \right) \ge 0. \eqno{(4)}$$

\noindent Indeed, by expanding the left hand side of $(4)$ and plugging in $(3)$, $(4)$ is proven.  Equation (4) is the monotonicity formula of Theorem \ref{mono}.

The story above generalizes in a straightforward way to hyperbolic space.  If $Y$ is a (k-1)-dimensional
surface in the boundary of the hyperbolic ball of radius $r$, then by symmetry we get $Vol_k CY = const(r,k) Vol_{k-1} Y$ for all $Y$.  We can
evaluate $const(r,k)$ by looking at the example when $CY$ is a hyperbolic k-dimensional ball of radius $r$.  We let $VHB_k(r)$ denote the volume of
a hyperbolic k-dimensional ball of radius $r$, and we let $VHS_{k-1}(r)$ denote the (k-1)-volume of the boundary of the ball.  We get

$$ Vol_k CY = \frac{VHB_k(r)}{VHS_{k-1}(r)} Vol_{k-1} Y. \eqno{(1H)}$$

We let $X$ be our minimal chain in $B^n_{hyp}(R)$, with $\partial X \subset \partial B^n_{hyp}(R)$.  We let $X_r:= X \cap B^n_{hyp}(r)$, and we 
let $Y_r = \partial X_r$.  By minimality, we get:

$$ Vol_k X_r \le Vol_k CY_r =  \frac{VHB_k(r)}{VHS_{k-1}(r)} Vol_{k-1} Y_r.  \eqno{(2H)} $$

As before, $\frac{d}{dr} Vol_k X_r \ge Vol_{k-1} Y_r$.  We let $V(r) := Vol_k X_r$.  The last equation becomes the following differential inequality:

$$ V'(r) \ge \frac{VHS_{k-1}(r)}{VHB_k(r)} V(r). \eqno{(3H)}$$

Since $\frac{d}{dr} VHB_k(r) = VHS_{k-1}(r)$, Equation (3H) implies the hyperbolic monotonicity formula from Theorem \ref{hypmono}:

$$ \frac{d}{dr} \left( \frac{V(r)}{VHB_k(r)} \right) \ge 0. \eqno{(4H)}$$

This concludes our description of some of the geometric ideas in the proof of Anderson's inequalities.  What is missing is to prove that $Z_{min}$ exists, and that $Z_{min}$ is smooth at almost every point, and that $Z_{min}$ is a sufficiently nice geometric object so that the reasoning above applies to it.  This is a standard topic in geometric measure theory, and it takes a substantial amount of work.  One can prove that $Z_{min}$ exists for cycles with coefficients in either $\bbz$ or $\bbz_2$, but the arguments are slightly different.  Working over $\bbz_2$ one can use mod 2 flat cycles instead of integral currents.  After proving the existence and some regularity of $Z_{min}$, the argument above will give the version of Anderson's theorem needed in the current paper.

\textsc{Larry Guth, MIT, Cambridge MA 02139 USA}

{\it lguth@math.mit.edu}

\bigskip

\textsc{Alexander Lubotzky, Institute of Mathematics, Hebrew University, Jerusalem 91904 Israel}

{\it alex.lubotzky@mail.huji.ac.il}

\end{document}